\documentclass[a4paper,11pt]{scrartcl}

\usepackage[english]{babel} 

\usepackage{amssymb,amsmath}
\usepackage{amsthm} 

\usepackage[T1]{fontenc}
\usepackage{bbm} 
\usepackage{enumerate}
\usepackage{graphicx}
\usepackage{float} 
\usepackage{enumitem} 
\usepackage{mathtools} 

\usepackage{cite} 
\usepackage{bm}

\usepackage[linkcolor=blue]{hyperref}
\hypersetup{
  colorlinks   = true, 
  urlcolor     = blue, 
  linkcolor    = blue, 
  citecolor   = red 
}

\newtheorem{theorem}{Theorem}
\newtheorem{prop}[theorem]{Proposition}
\newtheorem{lemma}[theorem]{Lemma}
\newtheorem{corollary}[theorem]{Corollary}

\theoremstyle{definition}
\newtheorem{definition}[theorem]{Definition}
\newtheorem{remark}[theorem]{Remark}

\newcommand{\Pac}{\mathcal{P}_{\text{ac}}(U)}
\newcommand{\txi}{\tilde{\xi}}
\newcommand{\tmu}{\tilde{\mu}}
\renewcommand{\SS}{\mathbb{S}}

\newcommand{\id}{\operatorname{id}}
\newcommand{\B}{\mathcal{B}(\mathbb{R})}
\newcommand{\N}{\mathbb{N}}
\newcommand{\R}{\mathbb{R}}

\newcommand{\txtd}{\textnormal{d}}
\newcommand{\txte}{\textnormal{e}}

\newcommand\numberthis{\addtocounter{equation}{1}\tag{\theequation}} 

\renewcommand{\epsilon}{\varepsilon}
\newcommand{\eps}{\epsilon}

\newcommand{\new}[1]{\textcolor{black}{#1}}

\let\oldmarginpar\marginpar
\renewcommand\marginpar[1]{\-\oldmarginpar[\scriptsize\raggedright\tiny\textcolor{black}{#1}]%
	{\raggedright\tiny\textcolor{black}{#1}}}

\DeclareMathOperator{\supp}{supp}

\DeclareMathOperator{\dive}{{\mathrm{div}}}

\begin{document}
\title{Global stability for McKean--Vlasov equations on large networks}
\author{Christian Kuehn\thanks{Technical University of Munich, Faculty of Mathematics, Boltzmannstra\ss e 3, 85748 Garching bei M\"unchen. E-mails: \texttt{\{christian.kuehn, tobias.woehrer\}@tum.de} \newline CK is supported by a Lichtenberg Professorship. TW is supported by the FWF under grant no. J 4681-N.} ~ and Tobias W\"ohrer$^*$}
\date{}
\maketitle
\abstract{
We investigate the mean-field dynamics of stochastic McKean differential equations with heterogeneous particle interactions described by large network structures. To express a wide range of graphs, from dense to sparse structures, we incorporate the recently developed graph limit theory of graphops into the limiting McKean--Vlasov equations. Global stability of the splay steady state is proven via a generalized entropy method, leading to explicit graph-structure dependent decay rates. We highlight the robustness of the entropy approach by extending the results to the closely related Sakaguchi-Kuramoto model with intrinsic frequency distribution. We also present central examples of random graphs, such as power law graphs and the spherical graphop, and analyze the limitations of the applied methodology.}

\section{Introduction}

Our analysis is motivated by \emph{stochastic McKean (or Kuramoto-type) differential equations} \cite{McKean1, McKean2}, which describe the behavior of $N$ particles which diffuse and interact with each other. We are specifically interested in heterogeneous interaction patterns that go beyond the case that all particles interact with all others. General interactions are then described by a graph/network structure with the corresponding equation system given as
\begin{equation}\label{eq:mckean} 
	\txtd X_t^i = -\frac{\kappa}{N r_N} \sum_{j\neq i}^N A^{ij} \nabla D(X_t^i - X_t^j)~ \txtd t + \sqrt{2}\, \txtd B_t^i, \quad i=1,\ldots, N,
\end{equation}
where $X_t^i$ is the $i$-th graph node with values on the $d$-dimensional flat torus of length $L>0$, denoted as $U:=[-\frac{L}{2}, \frac{L}{2}]^d$. The entries $A^{ij}\in\{0,1\}$ of the adjacency matrix $A^{(N)}:=(A^{ij})_{i,j=1,\ldots,N}$ represent the (undirected) graph edges and $1\geq r_N>0$ is a rescaling factor relevant for sparse graphs. The nodes are coupled along edges according to a (periodic) interaction potential $D:U\to \R$ with relative coupling strength $\kappa>0$. The stochasticity in the system is modeled by independent Brownian motions $B_t^i$ on $U$. Here we describe the dynamics of \eqref{eq:mckean} in the mean-field limit when the number of graph nodes tends towards infinity.

In the classical case of homogeneous interaction patterns, i.e.\ $A^{ij}= 1$ for all $i,j=1,\ldots N$ in \eqref{eq:mckean}, the mean-field limit (see \cite{oelschlager, Saka}) is given as the (homogeneous) \emph{McKean--Vlasov equation}
	\begin{align}\label{eq:mckeanPDE}
		\begin{aligned}
			\partial_t \rho &= \kappa \dive_x[\rho (\nabla_x D \star \rho)]+ \Delta_x \rho ,\quad t>0,\\
		\rho(0) &= \rho_0,
		\end{aligned}
	\end{align}
	where the solution $\rho(t,x)$ describes the average particle density at time $t\geq 0$ and position $x\in U$. Such nonlinear and nonlocal Fokker--Planck equations \cite{frank2005} have been analyzed from a wide variety of perspectives. Most prominently, synchronization phenomena for the Kuramoto model \cite{kuramoto75, kuramoto84, chiba2015} and its (noisy) mean-field formulations \cite{SSY88,strogatz91, crawford1994, strogatz2000, dietert2018} correspond to the potential $D(x) = -\cos(\frac{2\pi x}{L})$ in equation~\eqref{eq:mckeanPDE}. Other important examples are the Hegselmann-Krause model \cite{hegselmann02, BQQLC} in opinion dynamics and the  Keller-Segel model for bacterial chemotaxis \cite{keller1970}.	We refer to \cite{BQQLC, CaGyPaSc} for a recent treatment of a large range of interaction potentials (for homogeneous interactions), proving global stability via entropy methods and providing a bifurcation analysis.
 
	
	To incorporate heterogeneous interactions in the mean-field limit of \eqref{eq:mckean}, we require a theory of graph limits, which has made significant advances in recent years. For dense graphs the infinite node limits can be expressed as graph functions, called \emph{graph\-ons} \cite{LS06, borgs17}. A major shortcoming of this theory is that graphs with a subquadratic number of edge connections cannot be handled or the nodes are forced to have unbounded degree in the limit. Such graph cases require a refined treatment which has been developed for different subcases in a fragmented fashion. For sparse graphs of bounded degree there is the local convergence theory of Benjamini and Schramm \cite{BS01}, as well as the stronger local-global convergence theory for measure-based \emph{graphings} as limiting objects \cite{HLS2014, bollobas2011}. Sparse graphs of unbounded degree, such as power law graphs, which are most challenging but crucial for applications, have been treated via rescaled graphon convergence theory for $L^p$-graphons \cite{borgs2019}.
	
	These different graph limit frameworks have slowly started to be incorporated into dynamical mean-field settings. For various limiting equations solution theory and finite-graph approximation has been rigorously introduced and first graph limit dependent stability results have been shown \cite{ChMe1, ChMe2, KX2021, Med2014, med2017, lacker2023, lacker2019}. For non-dense graphs, there exists an appropriate choice of the rescaling factor $r_N$ in \eqref{eq:mckean} that results expressive limit. We refer to \cite{KX2021, med2019, borgs2019} for the detailed analysis.
	
	In this work we provide first entropy based global stability results for solutions to the mean-field limit of \eqref{eq:mckean}. To include a wide range of interaction structures --- from dense, to bounded and unbounded sparse graphs --- we formulate the mean-field problem in the very general graph limit theory of \emph{graphops} (or graph operators). Backhausz and Szegedy introduced this theory in \cite{BaSz20}, which unifies the above mentioned approaches by lifting them to the more abstract level of probability space based action convergence and operator theory. As we will see, this approach also has the advantage that it integrates rather naturally with the tools used in a PDE context. The mean-field formulation formally leads to the nonlinear \emph{graphop McKean--Vlasov equation} (see \cite{GJKM} for a discussion on the derivation)
	\begin{align}\label{eq:graphopmckean}
		\begin{aligned}
		\partial_t \rho &= \kappa \dive_x\Big(\rho V[A](\rho)\Big)+ \Delta_x \rho ,\quad t\geq 0,\\
		\rho(0) &= \rho_0,
	\end{aligned}
	\end{align}
where the solution $\rho(t,x,\xi)$ additionally depends on a graph-limit variable $\xi\in \Omega$ (where $(\Omega, \mathcal{A},\mu)$ is a Borel probability space specified below). In this abstract formulation, the heterogeneous interactions are expressed by a bounded linear operator $A$ (with the precise properties stated in Section~\ref{sec:het}), resulting in
	\begin{equation}\label{eq:vlasovgraphop}
	V[A](\rho)(t,x,\xi):= \int_{U} \nabla_x D(x-\tilde{x})(A\rho)(t, \tilde{x},\xi) ~\txtd \tilde{x},
	\end{equation}
 where we used the shorthand notation $(A\rho)(t, \tilde{x},\xi):=[A\rho](t, \tilde{x}, \cdot)(\xi)$, i.e., $A$ is an operator acting on the graph variable function space. Existence, and uniqueness have been established in the recent years for prototypical Vlasov-type equations and their approximations of dynamics on finite networks \cite{Ku2020, KX2021, GkKu, GJKM, med2019}. Yet, as far as the authors are aware, entropy based approaches have not yet been applied for any graphop PDE models. This work aims to provide a first application of the entropy method in this unifying graph operator setting. We emphasize the robustness of our method by also obtaining explicit global stability results for the closely related variant of \eqref{eq:graphopmckean} in 1D: the mean-field Sakaguchi-Kuramoto models with intrinsic frequency distributions combined with heterogeneous interactions. To our best knowledge, this is also the first entropy approach for mean-field Kuramoto models with heterogeneous interactions, including those involving graphons.\\

\noindent\textbf{General assumption:}
Throughout the paper, we always assume the interaction potential satisfies $D\in \mathcal{W}^{2,\infty}(U)$.

\subsection{Main results of the paper} Our main result is formulated in Theorem~\ref{thm:heterodecaygraphop}. There we assume reasonable regularity on a normalized initial datum and a square-integrable graphop $A$. Then, it states that classical solutions of the graphop McKean--Vlasov equation \eqref{eq:graphopmckean} converge exponentially fast to the splay (or `incoherent') steady state $\rho_\infty: = \frac{1}{L^d}$ in entropy, provided the interaction strength $\kappa$ is small enough. The bound on $\kappa$ as well as the exponential decay rate are explicit and they depend on the operator norm of the graphop.

	\subsection{Structure of the paper}
In Section~\ref{sec:hom} we review the case of global stability for homogeneous interactions and introduce the entropy method and necessary functional inequalities. In Section~\ref{sec:het} presents our main result of global stability for heterogeneous interaction patterns described by graphops. We further discuss the approaches limitations and the main theorem's formulation for graphons. 
In Section~\ref{sec:saka} we apply our method to the Sakaguchi-Kuramoto model with added frequency distribution.
Section~\ref{sec:examples} investigates explicit graphop examples such as spherical graphops and power law graphons.
	\section{Homogeneous interaction patterns}\label{sec:hom}
	We first recall that for all-to-all coupled interactions --- assuming $\kappa>0$ small enough --- the splay steady state $\rho_\infty:= \tfrac{1}{L^d}$ is a global steady state to which every initial configuration converges exponentially fast. This section we review the results developed in \cite{CaGyPaSc,BQQLC} on which the following sections are built.
	
	\subsection{Global stability via entropy methods}
	We start by revisiting the homogeneous interaction case by assuming $A\rho = \rho$ in \eqref{eq:graphopmckean}. This \new{is} the foundation for our strategy of proving decay estimates for more complex heterogeneous interactions. In the homogeneous case the network variable $\xi$ can be omitted, $\rho(t,x,\xi)\equiv  \rho(t,x)$ and $V[A](\rho) = \nabla_x D \star \rho$, i.e.\ we recover equation \eqref{eq:mckeanPDE}. We will introduce the entropy method and related functional inequalities that we will later extend and use for the heterogeneous interaction case.
	 \begin{definition}
	 	Let the space of absolutely continuous Borel probability measures on $U$ be denoted as $\mathcal{P}_{\text{ac}}(U)$. Then, for each $\rho\in\mathcal{P}_{\text{ac}}(U)$ we define the \emph{relative entropy functional}
	 \begin{equation}\label{eq:entropy}
	 	H(\rho|\rho_\infty):= \int_{U} \rho \log(\frac{\rho}{\rho_\infty}) ~\txtd x
	 \end{equation}
	 with $\rho_\infty(x) := \frac{1}{L^d}$.
	 	 \end{definition}
	 We use this functional to show exponential decay of solutions. The relative entropy is an adequate measure of ``distance'' between a function $\rho\in \mathcal{P}_{\text{ac}}(U)$ and the steady state $\rho_\infty$. Indeed, $H(\rho|\rho_\infty)\geq 0$ for all $\rho\in\Pac$, which can be seen by applying Jensen's inequality with
	  the convex function $x\log x$. It further holds that $H(\rho_\infty|\rho_\infty) = 0$ and the functional is linked to the $L^1(U)$ distance via
	 the \emph{Csisz\'ar-Kullback-Pinsker} (CKP) inequality (see \cite{BV05}) 
	 \begin{equation}\label{eq:CKP}
	 	\|\rho - \rho_\infty \|_{L^1(U)} \leq \sqrt{2 H(\rho|\rho_\infty)}.
	 \end{equation}
	  For functions $\rho\in\Pac$ with $\sqrt{\rho} \in H^1(U)$, it is bounded from above via the \emph{log-Sobolev inequality} (see e.g. \cite{EY87} for a direct proof)
	 \begin{equation}\label{eq:logsob}
	 	H(\rho|\rho_\infty) \leq \frac{L^2}{4\pi^2} \int_{U} |\nabla_x \log(\rho)|^2 \rho \,\txtd x.
	 \end{equation}
	 
	 For the existence of classical solutions to the homogeneous equation we refer to following result.
  
	 \begin{theorem}[\!\!{\cite[Theorem~2.2]{CaGyPaSc}} adapted from {\cite[Theorem~3.12]{BQQLC}}] For $\rho_0 \in H^{3 + d}(U) \cap \mathcal{P}_{\text{ac}}(U)$ there exists a unique classical solution $\rho \in C^1(0,\infty; C^2(U))$ of the homogeneous McKean--Vlasov equation \eqref{eq:mckeanPDE} such that $\rho(t,\cdot) \in \mathcal{P}_{\text{ac}}(U) \cap C^2(U)$ for all $t>0$. Additionally, it holds that $\rho(t,\cdot)>0$ and $H(\rho(t)|\rho_\infty)<\infty$ for all $t>0$.
	 \end{theorem}

\begin{prop}[{\!\!\cite[Proposition 3.1]{CaGyPaSc}}]\label{prop:homodecay}
Let the interaction potential of the homogeneous McKean--Vlasov equation \eqref{eq:mckeanPDE} satisfy $D\in \mathcal{W}^{2,\infty}(U)$. Let $\rho$ be the classical solution with initial data $\rho_0 \in H^{3 + d}(U) \cap \mathcal{P}_{\text{ac}}(U)$ and $H(\rho_0|\rho_\infty)<\infty$. If further the coupling coefficient satisfies
\begin{equation}\label{eq:kappacond}
	\kappa < \frac{2 \pi^2}{L^2 \| \Delta_x D \|_{L^\infty(U)} },
\end{equation}
then the solution $\rho$ is exponentially stable in entropy with the decay estimate
\begin{equation}\label{eq:Hdecay}
	H(\rho(t)|\rho_\infty) \leq  \txte^{-\alpha t} H(\rho_0|\rho_\infty), \quad t\geq 0,
\end{equation}
where
\begin{equation*}
	\alpha:= \frac{4 \pi^2}{L^2} - 2\kappa \| \Delta_x D\|_{L^\infty(U)} > 0.
\end{equation*}
\end{prop}

The detailed proof can be found in \cite[Proposition 3.1]{CaGyPaSc}. For later reference, we discuss the main steps in order to introduce the reader to the method and necessary inequalities that are generalized to the more involved heterogeneous interactions in Section~\ref{sec:het}.

\begin{proof}[Proof of Proposition \ref{prop:homodecay}] 
	Our goal is to bound the time derivative of a solution in relative entropy $H(\rho(t)| \rho_\infty)$  by a negative multiple of the relative entropy itself:
	\begin{equation}\label{eq:entineq}
		\frac{\txtd}{\txtd t}H(\rho(t)|\rho_\infty) \leq -\alpha H(\rho(t)|\rho_\infty), \quad t\geq 0,
	\end{equation}
with some $\alpha>0$.
Then Gronwall's Lemma provides us with an exponential decay rate $\alpha$ of solutions in relative entropy.
	
	To this end, we differentiate the relative entropy of solution $\rho$ with respect to time, which is possible as $\rho$ is a classical solution. Using equation \eqref{eq:mckeanPDE}, the general assumption $D\in \mathcal{W}^{2,\infty}(U)$ and integrating by parts leads to following two terms:
		\begin{align*} 
		\frac{\txtd}{\txtd t} H(\rho(t)|\rho_\infty) 	&= 
		\int_{U} \Big(\Delta_x \rho + \kappa \dive_x (\rho \nabla_x D \star \rho) )\Big) \log (\frac{\rho}{\rho_\infty})~\txtd x\\
		& \qquad + \underbrace{\int_{U} \rho \frac{\rho_\infty}{\rho} \frac{1}{\rho_\infty} \Big(\Delta_x \rho + \kappa \dive_x (\rho \nabla_x D \star \rho )\Big)~\txtd x}_{=0}
		\\
		&= - \int_{U} |\nabla_x \rho|^2\frac{1}{\rho}\, \txtd x - \kappa \int_{U} \rho \nabla_x D \star \rho (\frac{\rho_\infty}{\rho}) \nabla_x (\frac{\rho}{\rho_\infty})~\txtd x\\
		&=
		- \int_{U} |\nabla_x\log (\rho)|^2  \rho dx + \kappa  \int_{U} \rho (\Delta_x D \star \rho)~\txtd x.\numberthis \label{eq:dtHshort}
	\end{align*}
	The first term in \eqref{eq:dtHshort} can be estimated via the log-Sobolev inequality \eqref{eq:logsob}. Note that if the second term vanishes this estimate would provides us directly with an exponential decay result via Gronwall's lemma.
	
	To estimate the second term of interactions in \eqref{eq:dtHshort}, we replace $\rho$ by $\rho-\rho_\infty$ (possible due to $\int_{U} \Delta_x D(x) dx =  0$ following from the periodicity of $D$) and apply the CKP inequality \eqref{eq:CKP}
	to estimate:
	\begin{align*}
		\kappa \int_{U} \rho (\Delta_x D\star \rho) dx 
		&\leq \kappa \|\Delta_x D \|_{L^\infty(U)} \|\rho - \rho_\infty \|^2_{L^1(U)} \\
		&\leq 2\kappa  \|\Delta_x D \|_{L^\infty(U)}  H(\rho|\rho_\infty). \numberthis \label{eq:ckpest}
	\end{align*}
	
	In total, identity \eqref{eq:dtHshort} in combination with \eqref{eq:ckpest} yields
	\begin{equation}
		\frac{d}{dt} H(\rho|\rho_\infty ) \leq (-\frac{4 \pi^2}{L^2} + 2\kappa \|\Delta_x D\|_\infty) H(\rho|\rho_\infty),
	\end{equation}
	which proves the claimed result.
\end{proof}

\begin{remark}
	The condition on the coupling coefficient \eqref{eq:kappacond} can be further relaxed for coordinate-wise even interaction potentials, i.e.\ $D(\ldots,x_i,\ldots) =  D(\ldots, -x_i, \ldots)$ for $i=1,\ldots, d$. Then the term $\|\Delta_x D\|_{L^\infty(U)}$ can be replaced with $\|\Delta_x D_u\|_{L^\infty(U)}$, where $D_u$ is the interaction contribution of the Fourier coefficients with negative sign (see details in \cite{CaGyPaSc}).
\end{remark}

\begin{remark}\label{rem:classicalkuramoto}
	For $d=1$, that is $U = [-\frac L2, \frac L2]$, the noisy Kuramoto model\footnote{Note that the corresponding result in \cite{CaGyPaSc} includes the normalization term $\sqrt{\frac 2L}$ due to the definition via Fourier transform.} assumes an
	 interaction of form $D(x) = -\cos((\frac{2\pi}{L}) x)$. For this model, the above condition \eqref{eq:kappacond} simplifies to $\kappa < \frac12$.
\end{remark}

\section{Heterogeneous interaction patterns} \label{sec:het}
Before proving global stability for heterogeneous couplings, we first gather the necessary precise definitions for graphop McKean--Vlasov equations \eqref{eq:graphopmckean} and derived facts on graphops that are relevant to our setting. For an extensive introduction to graphops we refer to \cite{BaSz20}.

\begin{definition}\label{def:graphop}
~	\begin{itemize}
		\item 
	Let $(\Omega, \mathcal{A},\mu)$ be a Borel probability space with\footnote{If the support of $\mu$ is the whole set, a uniform-in-$\xi$ steady state is ensured.} $\supp{\mu} = \Omega$. Let further $L^p(\Omega):=L^p(\Omega,\mu)$, $p\in[1,\infty]$ be the corresponding real-valued Lebesgue spaces.
	
		\item We call \emph{$P$-operators} linear operators $A:L^\infty(\Omega) \to L^1(\Omega)$ with bounded norm
		\begin{equation*}
			\|A\|_{\infty \to 1}:= \sup_{v\in L^\infty(\Omega)}\frac{\|A v\|_{L^1(\Omega)}}{\|v\|_{L^\infty (\Omega)}}.
		\end{equation*}
		
		\item We call $P$-operators $A$ \emph{self-adjoint} if for all $f,g\in L^\infty(\Omega)$ it follows that
		\begin{equation*}
			(Af,g)_{L^2(\Omega)}:= \int_\Omega gAf~ \txtd\mu(\xi) = \int_\Omega f Ag~\txtd\mu(\xi)= (Ag,f)_{L^2(\Omega)}.
		\end{equation*}
	Note this definition of self-adjointness includes operators not necessarily acting on Hilbert spaces.
 \item A $P$-operator $A$ is $c$-\emph{regular} for some $c\in\R$ if $A1_\Omega = c 1_\Omega$.
		\item 
		An operator $A$ is said to be \emph{positivity preserving} if for any function $f\geq 0$ it follows that also $Af\geq 0$.
		
		\item \emph{Graphops} are positivity preserving and self-adjoint  $P$-operators.
		\item The graphop $A$ is associated with a \emph{graphon kernel} $W: \Omega\times \Omega\mapsto \R$, if
		\begin{equation}\label{eq:graphonA}
			(A\rho)(\xi):= \int_{\Omega} W(\xi,\tilde{\xi}) \rho(\tilde{\xi})~\txtd\mu(\tilde{\xi}).
		\end{equation}
		As $A$ is self-adjoint, it follows that the graphon is symmetric, i.e. $W(\xi,\txi) = W(\txi,\xi)$ for all $\xi,\txi \in \Omega$.
		
		\item The $P$-operator $A$ is of finite $(p,q)$-norm, where $1\leq p,q\leq \infty$, if the expression
		\begin{equation}\label{eq:pqnorm}
			\|A\|_{p \to q}:= \sup_{v\in L^\infty(\Omega)}\frac{\|A v\|_{L^{q}(\Omega)}}{\|v\|_{L^p (\Omega)}}
		\end{equation}
	is finite. By standard arguments, such an operator can be uniquely extended to a bounded operator $A:L^p(\Omega) \to L^q(\Omega)$, which again is a $P$-operator (or graphop if the original $A$ was). The set of such $P$-operators is denoted as $B_{p,q}(\Omega)$.
	\end{itemize}
\end{definition}

\begin{remark}
	As $(\Omega, \mathcal{A}, \mu)$ is a probability space we have $L^p(\Omega) \subseteq L^{p'}(\Omega)$ for $ 1\leq p' \leq p\leq \infty$. As a result, it holds true that $B_{p,q}(\Omega) \subseteq B_{p',q'}(\Omega) \subseteq B_{\infty,1}(\Omega)$ for $p'\geq p$, $q'\leq q$.
\end{remark}

\subsection{Existence of classical solutions}
The splay steady state of \eqref{eq:graphopmckean} is given by $\rho_\infty(x,\xi ):= \frac{1}{L^d}$. This means $\rho_\infty$  satisfies the stationary equation 
$$\kappa \dive_x\Big(\rho V[A](\rho)\Big)+ \Delta_x \rho = 0.$$
This can be seen as $\rho_\infty$ is constant in both $x$ and $\xi$ and thus
\begin{equation}
	\Delta_x D \star (A\rho_\infty)= A \rho_\infty \int_{U} \Delta_x D~\txtd x =  0
\end{equation}
where we again used the periodicity of the potential $D$.

\begin{definition}
For the probability space $(\Omega, \mathcal{A}, \mu)$ the \emph{relative entropy} (for heterogeneous coupling) is chosen as
\begin{equation}\label{eq:entropyxi}
	\hat{H}(\rho|\rho_\infty):= \int_{\Omega}\int_{U}  \rho \log(\frac{\rho}{\rho_\infty}) \, \txtd x ~\txtd\mu(\xi).
\end{equation}
\end{definition}
Observe that \eqref{eq:entropyxi} can be viewed as the average entropy over the heterogeneous node space. As for the homogeneous entropy $H(\rho|\rho_\infty)$ considered in \eqref{eq:entropy}, the CKP inequality \eqref{eq:ckpest} (now applied to the measure $\txtd x\times \txtd \mu$ on $U \times\Omega$) provides the lower bound
\begin{equation*}
	\|\rho - \rho_\infty \|_{L^1(U\times\Omega)} \leq \sqrt{2\hat{H}(\rho|\rho_\infty)}, \quad \rho \in \mathcal{P}_{\text{ac}}(U\times \Omega).
\end{equation*}
We assumed the measure of $L^1(\Omega,\mu)$ satisfies $\supp \mu = \Omega$, thus $\hat{H}(\rho|\rho_\infty ) = 0$ precisely when $\rho = \rho_\infty$.

\begin{definition}
We say an initial datum $\rho_0$ of \eqref{eq:graphopmckean} is \emph{admissible} if, for almost every $\xi\in\Omega$, it holds that $\rho_0(\cdot,\xi)\in H^{3+d}(U)\cap \Pac$. Furthermore, for each $x\in U$, we have $\rho_0(x,\cdot) \in L^\infty(\Omega,\mu)$.
\end{definition}

\begin{prop}\label{prop:existence}
	Let $\rho_0 $ be an admissible initial condition for the graphop McKean--Vlasov equation \eqref{eq:graphopmckean} with arbitrary graphop $A$. Then it follows that $\rho(\cdot,\cdot,\xi)$ is a unique classical solution\footnote{of equation \eqref{eq:graphopmckean} with $\xi$-fixed interaction term $V[A](\rho)(\cdot,\cdot,\xi)$).} for a.e.\ $\xi\in\Omega$. Furthermore, $\rho(t,\cdot,\xi) \in H^{3 + d}(U) \cap \mathcal{P}_{\text{ac}}(U)$, $\rho(t,\cdot,\cdot) \in \mathcal{P}_{\text{ac}}(U\times \Omega)$ and $\rho(t,\cdot,\cdot)>0$ for a.e.\ $\xi\in\Omega$ and all $t>0$.
\end{prop}
\begin{proof}
	We show that for each fixed $\xi\in\Omega$ and time $T>0$ the existence proof for a classical nonnegative solution of \cite{BQQLC} can be applied to our case. To see this, we show that the graphop $A$ still allows the necessary bounds for compactness arguments in norms w.r.t.\ $t$ and $x$. Consider the solution $\rho_n(t,x,\xi)$ for $n\in\N$ to the frozen linear equation with smooth initial data $\rho_0(x,\xi) \in \mathcal{P}_{\text{ac}}(U)\cap C^\infty(\overline{U})$:
	\begin{equation}\label{eq:frozeneq}
		\begin{cases}
			\partial_t\rho_n(t,x,\xi) = \Delta_x \rho_n(t,x,\xi) \\
			\qquad + \kappa \dive_x[\rho_{n}(t,x,\xi) \nabla_x D \star (A \rho_{n-1})(t,x,\xi) ], & t\in[0,T], x\in U,\\
			\rho_n(t,x,\xi) = \rho_n(t,x+Le_i,\xi), \quad t\in[0,T],& x\in \partial U,\\
			\rho_n(0,x,\xi) = \rho_0(x,\xi), &x\in U.
		\end{cases}
	\end{equation}
	For a.e.\ fixed $\xi\in\Omega$ the frozen equation can be solved with classical results of linear parabolic equations with bounded coefficients. This leads to a sequence of unique solutions $(\rho_n(\xi))_{n\in\N} \in \mathcal{P}_{\text{ac}}(U)\cap C^\infty(\overline{U} \times [0,T])$. We show that the sequence $(\rho_n)_{n\in\N}$ is bounded in the desired norms, by multiplying the linearized equation \eqref{eq:frozeneq} with $\rho_n(t,\xi)$ and integrating the $x$-variable. This leads to
	\begin{align*}
		\frac{\txtd}{\txtd t} \|\rho_n(t,\xi) \|^2_{L^2(U)} &+ 2\|\nabla_x \rho_n(t,\xi) \|^2_{L^2(U)} \\
		&\leq 2\kappa \int_U | \rho_n(t,\xi) \nabla_x D \star A\rho_{n-1}(\xi) \nabla_x \rho_n(t,\xi) | dx\\
		&\leq \frac{\kappa \eps^2}{2} \|\nabla_x \rho_n(t,\xi) \|^2_{L^2(U)} \\
		&\qquad + \frac{2\kappa }{\eps^2} \|\rho_n(t,\xi) \|^2_{L^2(U)} \|\nabla_x D \star A \rho_{n-1}(t,\xi) \|^2_{L^\infty(U)}\numberthis \label{eq:epsest}.
	\end{align*}
	To estimate the last term above, we use the fact that $A$, as a linear bounded operator acting solely on the network variable, commutes with the $x$-integral and that the frozen equation \eqref{eq:frozeneq} is mass preserving, i.e.\ $\|\rho_n(t,\xi)\|_{L^1(U)} = 1$ for $a.e.~\xi \in \Omega$ and all $t>0$:
	\begin{align*}
		\|\nabla_x D \star A \rho_{n-1}(t,\xi) \|^2_{L^\infty(U)} &\leq 
		\|\nabla_x D \|^2_{L^\infty(U)} \| A \rho_{n-1}(t,\xi) \|^2_{L^1(U)} \\
		&= \|\nabla_x D\|^2_{L^\infty(U)} (A  \| \rho_{n-1}(t,\xi) \|_{L^1(U)})^2 \\
		&= \|\nabla_x D\|^2_{L^\infty(U)} (A 1_\Omega(\xi))^2.\numberthis \label{eq:estimateonstar}
	\end{align*}
	Inserting this estimate in \eqref{eq:epsest} and choosing  $\eps = (2\kappa)^{-\frac12}$, we obtain
	\begin{align}
		\begin{aligned}\label{eq:dtrhon}
		\frac{d}{dt} \|\rho_n(t,\xi) \|^2_{L^2(U)} +& \|\nabla_x \rho_n(t,\xi) \|^2_{L^2(U)} \\
		&\leq \kappa^2 \|\nabla_x D\|^2_{L^\infty(U)} (A 1_\Omega(\xi))^2 \|\rho_n(t,\xi) \|^2_{L^2(U)}. 
		\end{aligned}
	\end{align}
	Gronwall's Lemma yields the upper bound
	\begin{equation}\label{eq:A1}
		\|\rho_n(t,\xi) \|^2_{L^2(U)} \leq C(\xi,T) \|\rho_0(\xi) \|^2_{L^2(U)},\quad n\in\N,
	\end{equation}
with 
\begin{equation*}
	C(\xi,T):= \exp\left(\kappa^2 \|\nabla_x D \|_{L^\infty(U)}^2 (A1_\Omega(\xi))^2 T\right),
\end{equation*}
	which is finite for $a.e.\ \xi\in\Omega$. Integrating \eqref{eq:dtrhon} w.r.t.\  the time-variable and using the bound \eqref{eq:dtrhon} yields a uniform-in-$n$ bound for $(\rho_n(\xi))_{n\in\N}$ in $L^2(0,T,H^1(U))$. The further bootstrapping steps for initial datum of general regularity, uniqueness, desired solution regularity and positivity of solutions follow as described in \cite{CaGyPaSc, BQQLC}.
	Let us comment on the specific requirement $\rho_0 \in H^{3+d}(U)$: It is needed in order to show, using the equation's specific structure, that $\partial_{t}\rho \in L^2(0,\infty;H^{2+d}(U))$. As the embedding $H^{2+d}(U) \hookrightarrow H^{1+d}(U)$ is compact, it follows by Aubin-Lions Lemma that $\partial_t \rho \in C(0,\infty;H^{1+d}(U))$. Using the compactness of the Sobolev embedding one more time yields the desired $\partial_t \rho \in C(0,\infty; C(U))$.
	
	For each $t>0$ and a.e.\ $\xi\in\Omega$ the solution is mass preserving, hence $\rho(t,\cdot,\xi)\in \mathcal{P}_{\text{ac}}(U)$. As 
	\begin{equation*}
		\int_\Omega \int_U \rho(t,x,\xi) ~\txtd x ~\txtd\mu(\xi) = \int_\Omega 1~\txtd\mu(\xi) = 1,
	\end{equation*}
 Fubini-Tonelli's theorem implies that $\rho(t,\cdot,\cdot) \in \mathcal{P}_{\text{ac}}(U\times \Omega)$.	
%
%
\end{proof}
If we add further assumptions on the graphop, we obtain \emph{some} regularity in the $\xi$-variable.

\begin{corollary}\label{cor:existence}
	Let $\rho_0$ be an admissible initial datum for the graphop McKean--Vlasov equation \eqref{eq:graphopmckean} with graphop $A$.
	If $A\in B_{2,2}(\Omega)$ holds and $\hat{H}(\rho_0|\rho_\infty) < \infty$ then  $\hat{H}(\rho(t)|\rho_\infty)<\infty$ for all $t>0$. If $A1_\Omega(\xi) \leq c$ for almost every $\xi\in\Omega$ with some constant $c\geq 0$, then for all $t>0$ it holds that $\rho(t, \cdot,\xi) \in H^{3+d}(U)\cap \mathcal{P}_{\text{ac}}(U)$ for a.e.\ $\xi \in \Omega$ and $\rho(t, x,\cdot)\in L^\infty(\Omega)$ for all $x\in U$.
\end{corollary}

\begin{proof}
\new{The proof of $\hat{H}(\rho(t)|\rho_\infty)<\infty$ for all $t>0$, assuming} $A\in B_{2,2}(\Omega)$ is deferred to the proof of Theorem~\ref{thm:heterodecaygraphop} as the steps are identical but finiteness works for arbitrary $\kappa\geq 0$.

If $A1_\Omega (\xi)\leq c$, then the estimate \eqref{eq:estimateonstar} can be refined further by
\begin{align*}
	\|\nabla_x D \star A \rho_{n-1}(\xi) \|^2_{L^\infty(U)} &\leq  \|\nabla_x D\|^2_{L^\infty(U)} c^2.
\end{align*}
Plugging this into \eqref{eq:A1} results in
\begin{equation*}
	\sup_{\xi\in\Omega}  \|\rho_n(t,\xi) \|^2_{L^2(U)} \leq \exp\left[c^2 \|\nabla_x D \|^2_{L^\infty(U)} t\right] \sup_{\xi\in\Omega} \|\rho_0(\xi ) \|_{L^2(U)}
\end{equation*}
and proves the claim.
\end{proof}
\begin{remark}
    The assumption $A1_\Omega(\xi) \leq c$ of Corollary~\ref{cor:existence} clearly includes $c$-regular graphops with any $c\geq 0$.
\end{remark}

\begin{remark}
	The regularity results of Proposition~\ref{prop:existence} and Corollary~\ref{cor:existence} with regards to the network variable $\xi$ are expected to be improvable. One would hope that for the weaker assumption $\|A\|_{2\to2}<\infty$ that the solutions remain $L^2(\Omega)$-integrable for all times. However, the above estimates \eqref{eq:dtrhon} do not provide such a bound. Part of the difficulty is rooted in the fact that the evolution equation \eqref{eq:graphopmckean} provides no explicit regularizing term over time in the $\xi$-variable and there is no mixing (and not even a geometrical link) of the $x$ and $\xi$ directions. This is reminiscent of a parabolic equation with degenerate diffusion (in $\xi$) for which no hope of a coupling mechanism is present \cite{villani09}.
 
 Similar existence analysis for the nonlinear heat equation on sparse graphs \cite[Section~3]{med2017} and the Kuramoto model \cite{med2019} show $L^2(\Omega)$ regularity of solutions. However, also these existence results are restricted to the cases of $c$-regular graphops.
\end{remark}
\subsection{Global stability}
Before establishing the global convergence result for solutions to \eqref{eq:graphopmckean} in entropy \eqref{eq:entropyxi}, we recall the definition of the numerical radius of an operator. This quantity allows the sharpest formulation of the convergence rate which is possible with our method.
\begin{definition}
\item The \emph{numerical radius} of a graphop $A\in B_{2,2}(\Omega)$ is given as
\begin{equation*}
	n(A):=\sup \{(Af,f) \mid f\in L^2(\Omega), \|f\|_{L^2(\Omega)} = 1 \}.
\end{equation*}
\end{definition}
For graphops in $B_{2,2}(\Omega)$ the numerical radius is an equivalent norm to the operator norm\footnote{The second inequality is generally wrong on real Hilbert spaces. However for the extension of $L^2(\Omega)$ to a complex Hilbert space it can be shown by the Polarization Identity. As graphops are symmetric, the resulting upper bound remains valid on the restriction to the real-valued Hilbert space $L^2(\Omega)$. Additionally, the restriction yields the desired estimates as the complex Hilbert space numerical radius of a symmetric operator is identical with its real Hilbert space numerical radius, see \cite{brickman1961}. } with bounds:
\begin{equation}\label{eq:numericalequi}
	n(A) \leq \|A\|_{2\to2} \leq 2 n(A).
\end{equation}

\begin{theorem}\label{thm:heterodecaygraphop}
	Consider the graphop McKean--Vlasov equation \eqref{eq:graphopmckean} with any graphop $A$ that satisfies $n(A) < \infty$. Let $\rho_0$ be any admissible initial datum with $\hat{H}(\rho_0|\rho_\infty)<\infty$. Let further the coupling coefficient satisfy
	\begin{equation}\label{eq:kappacondgraphop}
		\kappa < \frac{2 \pi^2}{L^2 \| \Delta_x D \|_{L^\infty(U)} n(A)  },
	\end{equation}
Then, the classical solution $\rho$ is exponentially stable with the decay estimate
\begin{equation}\label{eq:Hdecaygraphop}
		\hat{H}(\rho(t)|\rho_\infty) \leq  \txte^{-\hat{\alpha}(A) t} \hat{H}(\rho_0|\rho_\infty), \quad t\geq 0,
	\end{equation}
	where
	\begin{equation*}
		\hat{\alpha}(A):= \frac{4 \pi^2}{L^2} - 2\kappa  \| \Delta_x D\|_{L^\infty(U)}n(A) > 0.
	\end{equation*}
\end{theorem}

\begin{proof}
	Our strategy is to generalize the proof of Proposition~\ref{prop:homodecay}.
	
	Let us first note that, as $\|\rho(t,\xi)\|_{L^1(U)} =1$ for a.e.\ $\xi\in\Omega$ and $t>0$, the mapping $\xi \mapsto \|\rho(t,\xi) - \rho_\infty \|_{L^1(U)}$ has finite $L^p(\Omega)$-norm, for any $p\in[1,\infty]$. Specifically, $\xi \mapsto \|\rho(t,\xi) - \rho_\infty \|_{L^1(U)} \in L^2(\Omega)$.
	
	As $\rho(\xi)$ is a classical solution for a.e.\ $\xi$, it follows that the mapping $(t\mapsto H(\rho(t)|\rho_\infty)) \in C^1(0,\infty)$. Then, $(t\mapsto\hat{H}(\rho(t)|\rho_\infty))\in C^1(0,\infty)$ follows from the calculation below together with the $L^\infty(\Omega)$ bound noted above.  The same calculations as for \eqref{eq:dtHshort}, and again using the log-Sobolev inequality \eqref{eq:logsob} (for the  measure  $\txtd x\times \txtd\mu$), yields for $t\geq 0$
	\begin{equation}\label{eq:dtHhetero}
		\frac{d}{dt} \hat{H}(\rho|\rho_\infty) \leq - \frac{4\pi^2}{L^2}\hat{H}(\rho|\rho_\infty) + \kappa  \int_{\Omega}\int_{U} \rho [\Delta_x D \star (A\rho)] ~\txtd x ~\txtd\mu(\xi).
	\end{equation}
	
	To estimate the second term in \eqref{eq:dtHhetero}, we use the fact that $A$, as a linear bounded operator acting solely on the network variable, commutes with the $x$-integral:
		\begin{equation}
		\Delta_x D \star (A\rho)= A\,(\Delta_x D \star \rho).
	\end{equation}
For the special case $\rho \equiv \rho_\infty$, we even have
	\begin{equation}
		\Delta_x D \star (A\rho_\infty)= A \rho_\infty \int_{U} \Delta_x D ~\txtd x =  0,
	\end{equation}
where the last equality follows again from the periodicity of $D$.	Hence, we can replace both occurrences of $\rho$ by $\rho-\rho_\infty$ and use H\"older's inequality in $U$ with $p = 1$ and $p^* = \infty$ to estimate
	\begin{align*}
		\kappa \int_{\Omega}&\int_{U} \rho (\Delta_x D \star (A\rho) )~ \txtd x ~\txtd\mu(\xi) \\
		&\leq \kappa \|\Delta_x D \|_{L^\infty(U)} \int_{\Omega}  \|A[\rho - \rho_\infty]\|_{L^1(U)} \|\rho - \rho_\infty \|_{L^1(U)}~\txtd\mu(\xi).\numberthis \label{eq:secondtermest}
	\end{align*}
	Given that a graphop $A$ preserves positivity and denoting $\nu_\pm = \mp\max\{0,\pm \nu\}$ it follows that $$|A\nu| = |A\nu_+ - A\nu_-| \leq |A\nu_+| + |A\nu_-| = A\nu_+ + A\nu_- = A|\nu|.$$
	With this, we can estimate \eqref{eq:secondtermest}, use the definition of the numerical radius of $A$, and then apply the CKP inequality  \eqref{eq:CKP} for a.e.\ $\xi\in \Omega$. This leads to
	\begin{align*}
		\kappa \int_{\Omega}&\int_{U} \rho (\Delta_x D \star (A\rho) ) ~\txtd x~\txtd\mu(\xi)\\
		&\leq \kappa \|\Delta_x D \|_{L^\infty(U)} \int_{\Omega}  (A\|\rho - \rho_\infty\|_{L^1(U)}) \|\rho - \rho_\infty \|_{L^1(U)} ~\txtd\mu(\xi)\\
	&\leq \kappa \|\Delta_x D \|_{L^\infty(U)} n(A)\int_{\Omega}  \|\rho - \rho_\infty\|_{L^1(U)}^2 ~\txtd\mu(\xi)\\
	&\leq 2\kappa \|\Delta_x D \|_{L^\infty(U)} n(A)\hat{H}(\rho|\rho_\infty).
\end{align*}
Finally, applying Gronwall's lemma to the estimation of \eqref{eq:dtHhetero} leads to the desired result.
\end{proof}

\begin{remark}	
	Theorem~\ref{thm:heterodecaygraphop} naturally includes graphops that satisfy the (stronger) condition $\|A\|_{p\to q}< \infty$ with $p<2$, $q>2$. In such cases the conditions of Theorem~\ref{thm:heterodecaygraphop} are still met since $n(A)\leq  \|A\|_{2\to 2} \leq \|A\|_{p\to q}< \infty$. But compared to the $(p,q)$-norms for $p\leq 2, q\geq 2$, the numerical radius always gives the sharpest estimates with regard to our line of estimations.
\end{remark}

\subsection{Limitations of the entropy method approach}\label{subsec:noL2onA}

	Generalizing the exponential stability of Theorem~\ref{thm:heterodecaygraphop} to graphops that do not have finite $\|A\|_{2\to 2}$ norm seems not feasible with our method of proof. This is due to the necessity of relating the time-derivative of the defined entropy back to the entropy itself \eqref{eq:entineq}. \new{Nonetheless, below we show that in these cases, solutions at least remain bounded for all times and are therefore \emph{not} exponentially unstable.}
 
 Let us look at the considerations in detail:
	\begin{enumerate}[label=(\roman*)]
		\item \label{item:noCA} For exponential decay, the proof of Theorem~\ref{thm:heterodecaygraphop} requires the existence of a constant $C(A)<\infty$ such that 
		\begin{align}\label{eq:CAestimate}
			\begin{aligned}
			\int_{\Omega} ( A\|\rho - \rho_\infty\|_{L^1(U)})& \|\rho - \rho_\infty\|_{L^1(U)} ~\txtd\mu(\xi)\\
			&\leq C(A) \int_{\Omega}  \|\rho - \rho_\infty\|_{L^1(U)}^2 ~\txtd\mu(\xi).
			\end{aligned}
		\end{align}
	This is necessary\footnote{The other option is to apply the CKP inequality to $\|\rho-\rho_\infty\|^2_{L^1(U\times\Omega)}$ for $\txtd x\times \txtd\mu$. This would also close the entropy inequality. But as $\|\rho-\rho_\infty\|^2_{L^1(U\times \Omega)} \leq \int_\Omega \|\rho-\rho_\infty \|^2_{L^1(\Omega)}~\txtd\xi$ one only obtains a constant $\tilde{C}(A)<\infty$ such that $C(A)\leq \tilde{C}(A)$.} to subsequently apply the CKP inequality \eqref{eq:CKP} for a.e.\ $\xi\in\Omega$  and close the differential inequality. But if $A$ is unbounded in $L^2(\Omega)$ no such constant can exist:
	
	Let us first assume $\|A\|_{2\to2}<\infty$. Then the numerical radius $n(A)$ is finite as well, due to the norm equivalence \eqref{eq:numericalequi}. From the definition of the numerical radius follows that $C(A) = n(A)$ is optimal in \eqref{eq:CAestimate}.
	
	Now if $A$ is an unbounded operator in $L^2(\Omega)$ with $\|A\|_{2\to2} = \infty$ this implies that the numerical radius $n(A)$ is also unbounded and hence no constant $C(A)<\infty$ exists to bound \eqref{eq:CAestimate}. This follows again from \eqref{eq:numericalequi}, as we can approximate $A$ by a sequence of bounded operators $A_n$ with growing operator norm such that $\lim_{n\to \infty} \|A\rho-A_n\rho\|_{L^2(\Omega)} = 0$ for all $\rho\in \mathcal{D}(A)\subseteq L^2(\Omega)$.

		\item If $\|A\|_{2\to 2}$ is unbounded one could still consider a H\"{o}lder inequality estimate.

	For the case that only the weaker condition $\|A\|_{p\to 2}<\infty$ for some $p>2$ is satisfied, the Cauchy-Schwarz inequality and $\|A\rho\|_{L^2(\Omega)} \leq \|A\|_{p\to 2} \|\rho\|_{L^p(\Omega)}$ yields 
		\begin{align*} \begin{aligned}
			&\int_{\Omega} ( A\|\rho - \rho_\infty\|_{L^1(U)}) \|\rho - \rho_\infty\|_{L^1(U)} ~\txtd\mu(\xi)\leq \\
			&\quad \|A\|_{p\to 2} \left(\int_{\Omega}  \|\rho - \rho_\infty\|_{L^1(U)}^{p} ~\txtd\xi\right)^{\frac1{p}}\left(\int_{\Omega}  \|\rho - \rho_\infty\|_{L^1(U)}^{2}~\txtd\mu(\xi)\right)^{\frac1{2}}.
		\end{aligned}
	\end{align*}
However, as $p>2$ the $L^p(\Omega)$-term cannot be bounded by an $L^2(\Omega)$-term and hence the differential inequality cannot be closed by the CKP inequality.

Similarly, if only the condition $\|A\|_{2\to q}<\infty$ for some $q<2$ holds true, then H\"{o}lder's inequality with $\frac1q + \frac1{q^*} = 1$ leads to a $L^{q^*(\Omega)}$-norm factor and as $q^* >2$ again it cannot be bounded by $L^2(\Omega)$.

\item Let us consider to modify the entropy functional \eqref{eq:entropyxi}. Instead of integrating $\xi$ with respect to the underlying measure of the probability space $(\Omega,\mathcal{A}, \mu)$, we could consider an additional probability Borel measure $\tilde{\mu}$ on $\Omega$:
\begin{equation}\label{eq:entropymod}
\hat{H}_{\tilde{\mu}}(\rho|\rho_\infty):=\int_\Omega \int_U \rho \log(\frac{\rho}{\rho_\infty})~\txtd x~\txtd\tilde{\mu}(\xi).
\end{equation}
Then closing of the differential inequality for $H_{\tilde{\mu}}(\rho|\rho_\infty)$ along the proof of Theorem \ref{thm:heterodecaygraphop} requires a constant $\tilde{C}(A)<\infty$ such that
	\begin{align}\label{eq:CAtilde}
	\begin{aligned}
		\int_{\Omega} ( A\|\rho - \rho_\infty\|_{L^1(U)}) &\|\rho - \rho_\infty\|_{L^1(U)}~\txtd\tilde{\mu}(\xi)
	\\ &\leq \tilde{C}(A) \int_{\Omega}  \|\rho - \rho_\infty\|_{L^1(U)}^2 ~\txtd\tilde{\mu}(\xi).
	\end{aligned}
\end{align}
This is possible if and only if $A$ has a bounded numerical radius in the $L^2(\Omega, \tilde{\mu})$ sense. For the counter example of power law graphons without finite $L^2(\Omega)$ operator norm, we show in \S \ref{subsec:powerlaw} that no reasonable measure $\tmu$ can help.

Note that in general $A$ is not necessarily a graphop in $L^2(\Omega,\tilde{\mu})$ as it might not even be symmetric in $L^2(\Omega,\tilde{\mu})$. Thus, the numerical radius equivalence \eqref{eq:numericalequi} is not guaranteed, however the lower bound $r_{\tilde{\mu}}(A)\leq \|A\|_{2\to2,\tilde{\mu}}$ still holds.

\item The above does not exclude the possibility that completely different approaches can lead to stability results for more general graphops. For example a different entropy functional or a different structuring of the graph and position dependent interaction term in specific cases.
\end{enumerate}

If we drop the goal of exponential decay, we can still infer that the steady state $\rho_\infty$ is \emph{not} exponentially unstable. 
\begin{corollary}\label{cor:solbounded}
    Let the assumptions and notation of Theorem~\ref{thm:heterodecaygraphop} be given but let $A$ be an arbitrary graphop. Then, classical solutions $\rho$ to \eqref{eq:graphopmckean} \new{are bounded for all times, i.e.\
    \begin{equation}
        \hat{H}(\rho(t)|\rho_\infty) \le c,\quad t\geq 0
    \end{equation}
    with a constant $c\geq 0$ which depends on the equation coefficients and $\hat{H}(\rho_0|\rho_\infty).$}
\end{corollary}
\begin{proof}
The proof of Theorem~\ref{thm:heterodecaygraphop} yields the estimate
\begin{align}
\begin{aligned}\label{eq:dHunstable}
  \frac{d}{dt} & \hat{H}(\rho|\rho_\infty) \leq   -\frac{4\pi^2}{L^2} \hat{H}(\rho|\rho_\infty)\\
  &  +  \kappa \|\Delta_x D \|_{L^\infty(U)} \int_{\Omega}  (A\|\rho - \rho_\infty\|_{L^1(U)}) \|\rho - \rho_\infty \|_{L^1(U)} ~\txtd\mu(\xi).
  \end{aligned}
\end{align}

In the following estimation we use the fact that the graphop $A$ satisfies $\|A\|_{\infty \to 1}<\infty$, that $\|\rho(\xi) - \rho_\infty \|_{L^1(U)} \leq 2$ for all $\xi\in\Omega$ due to $\rho(t,\cdot,\xi),\rho_0\in \Pac$ and the CKP inequality \eqref{eq:CKP}:
\begin{align*}
			&\int_{\Omega} ( A\|\rho - \rho_\infty\|_{L^1(U)}) \|\rho - \rho_\infty\|_{L^1(U)} ~\txtd\mu(\xi)\\
			&\quad \leq \|A\|_{\infty\to 1} \sup_{\xi\in\Omega} \|\rho - \rho_\infty\|_{L^1(U)}\left(\int_{\Omega}  \|\rho - \rho_\infty\|_{L^1(U)}~\txtd\mu(\xi)\right)
   \\
			&\quad \leq 2\|A\|_{\infty\to 1} \left(\int_{\Omega}  \|\rho - \rho_\infty\|_{L^1(U)}~\txtd\mu(\xi)\right)
    \\
			&\quad \leq \sqrt{8}\|A\|_{\infty\to 1} \sqrt{\hat{H}(\rho|\rho_\infty)}.\numberthis \label{eq:CKPunstable}
	\end{align*}
\new{Plugging \eqref{eq:CKPunstable} into \eqref{eq:dHunstable} yields
\begin{equation}
 \frac{d}{dt} \hat{H}(\rho|\rho_\infty) \leq -\frac{4\pi^2}{L^2} \hat{H}(\rho|\rho_\infty) + \sqrt{8}\kappa \|\Delta_x D \|_{L^\infty(U)}\|A\|_{\infty\to 1} \sqrt{\hat{H}(\rho|\rho_\infty)}.
\end{equation}
Excluding the trivial case $\rho(t) = \rho_\infty$, we denote $v(t):= \sqrt{H(\rho(t)|\rho_\infty)} > 0$ and the non-negative constants $a:=\frac{4\pi^2}{L^2}$ and $b:= \sqrt{8}\kappa \|\Delta_x D \|_{L^\infty(U)}\|A\|_{\infty\to 1}$, we obtain the differential inequality
\begin{equation*}
	\dot{v}(t) \leq -\frac{a}{2} v(t) + b, \quad t\geq 0.
\end{equation*}
If for any $t\geq 0$ the right hand side is positive, this is equivalent to the bound $v(t) < \frac{2b}{a}$. For any other $t\geq 0$, $v$ is non-increasing, thus $v(t) \leq \max\{v(0), \frac{2b}{a}\}$ for all $t\geq 0$. As $v(t)^2 = \hat{H}(\rho(t)|\rho_\infty)$ the claimed result follows.}
\end{proof}
The result of Corollary~\ref{cor:solbounded} tells us that we are not too far from a global stability result for general graphops. This gives hope for an extension with an appropriate method.

\subsection{Global stability for graphons}\label{subsec:graphons}

Let us now discuss the special cases of graphops in the form of integral operators with associated graphons (cf. Definition~\ref{def:graphop}). 

\begin{definition}
	The \emph{graphon norm} is defined as
	\begin{equation}\label{eq:Wpnorm}
	\|W\|_p:= 
	\begin{cases}
		\left(\int_{\Omega}\int_{\Omega} W(\xi,\txi)^{p}~\txtd\txi ~\txtd\mu(\xi) \right)^{\frac1{p}},& p\in[1,\infty),\\
		\sup_{(\xi,\txi) \in \Omega^2} |W(\xi,\txi)|,& p=\infty.
	\end{cases}
\end{equation}
\end{definition}

Given a graphon $W$ that satisfies $\|W\|_p<\infty$, it follows that for the associated graphop $A_W$ we have $A_W\in B_{p,p^*}(\Omega)$ with $\|A\|_{p\to p^*} \leq \|W\|_p$. Additionally, for $p>1$ the operator $A_W$ is compact. Indeed, the boundedness follows by applying H\"{o}lder's inequality for $\frac1p + \frac1{p^*} = 1$:
\begin{align}\begin{aligned}
		\label{eq:graphonoperatornorm}
	\|A_W f \|^{p^*}_{p^*} &= \int_\Omega \left(\int_\Omega W(\xi,\txi) f(\txi) ~\txtd\mu(\txi)\right)^{p^*} \txtd\mu(\xi)\\
	&\leq  \|W\|_{p^*}^{p^*} \left(\int_\Omega f(\txi)^{p} ~\txtd\mu(\txi) \right)^{\frac{p^*}{p}}.
	\end{aligned}
\end{align}
The compactness of $A_W$ for the case $\|W\|_{p^*}<\infty$ for $p^*>1$ follows by finite rank approximation, e.g.\ an approximation of $W$ with polynomials which are dense in $L^{p^*}(\Omega\times \Omega)$.

Graphops that have a graphon density satisfying $\|W\|_{2}<\infty$ are specifically convenient to treat due to the underlying Hilbert space structure. As (positivity-preserving) \emph{Hilbert-Schmidt operators} their spectrum consists exclusively of the (positive) point spectrum, $\sigma(A)= \sigma_p(A)$, with the only possible accumulation point at $0$. Specifically, the spectral radius $\operatorname{rad}(A)$ of $A$ coincides with its numerical radius and bounds the operator norm:
\begin{equation*}
	\operatorname{rad}(A_W):=\sup\{|\lambda| \mid \lambda \in \sigma(A_W)\} = \lambda_{\max}^{A_W} = n(A_W)
\end{equation*} 
where $\lambda_{\max}^{A_W}$ denotes the largest eigenvalue of $A_W$. 

Theorem~\ref{thm:heterodecaygraphop} directly implies the following:

\begin{corollary}\label{cor:heterodecay}
	Let the assumptions and notations of Theorem~\ref{thm:heterodecaygraphop} be given. Furthermore, let the graphop $A_W$ in \eqref{eq:graphopmckean} be associated to a graphon density $W:\Omega \times \Omega \to \R$, such that $\|W\|_{2}<\infty$. If the coupling coefficient satisfies
	\begin{equation}\label{eq:kappagraphon}
		\kappa < \frac{2 \pi^2}{L^2 \| \Delta_x D \|_{L^\infty(U)} \lambda_{\max}^{A_W}  },
	\end{equation}
then the solution $\rho$ to equation \eqref{eq:graphopmckean} is exponentially stable with the decay estimate
	\begin{equation}\label{eq:Hdecaygraphon}
		\hat{H}(\rho(t)|\rho_\infty) \leq  \txte^{-\hat{\alpha}(W) t} \hat{H}(\rho_0|\rho_\infty), \quad t\geq 0,
	\end{equation}
	where
	\begin{equation*}
		\hat{\alpha}(W):= \frac{4 \pi^2}{L^2} - 2\kappa  \| \Delta_x D\|_{L^\infty(U)}\lambda_{\max}^{A_W} > 0.
	\end{equation*}
\end{corollary}
\begin{proof}
The result follows directly from Theorem~\ref{thm:heterodecaygraphop} and \eqref{eq:graphopHS}.
\end{proof}

\begin{remark}
 As $A_W$ is symmetric in $L^2(\Omega)$ it follows that
\begin{equation}\label{eq:graphopHS}
\lambda_{\max}^{A_W} = n(A_W) = 	\|A_W\|_{2\to 2}  \leq  \|W\|_2.
\end{equation}   
Hence, we can use $\lambda_{\max}^{A_W} \leq \|W\|_2$ in the estimates of Corollary~\ref{cor:heterodecay}, which leads to less optimal but more accessible coupling conditions and decay estimates.
\end{remark}

\section{\new{Sakaguchi-Kuramoto model with frequency distribution}}\label{sec:saka}
Let us consider global stability of the splay state for a well-studied variant of the McKean--Vlasov equation \eqref{eq:mckeanPDE}: The \emph{Sakaguchi-Kuramoto mean-field equation} \cite{Saka} with intrinsic frequency distribution $g$. It is given as
\begin{equation*}
	\begin{aligned}\label{eq:sakaguchi}
		\partial_t \rho &= \partial_x(-\omega \rho + \kappa \rho V[A,g](\rho)) + \beta^{-1}\partial_{xx} \rho,\quad  t\geq 0,\\
		\rho(0) &= \rho_0,
	\end{aligned}
\end{equation*}
where a frequency-dependent transport term is added. The Vlasov term is dependent on the frequency density function $g$ via 
\begin{equation}\label{eq:freqvlasov}
	V[A,g](\rho):= \int_\R (\nabla D * A\rho) ~g \txtd\omega.
\end{equation}
We also explicitly include an inverse temperature parameter $\beta>0$ to the diffusion term in order to discuss the limit $\beta\to0+$.  In this case, solutions depend on $\rho(t,x,\xi,\omega)$ with $x\in [-\pi,\pi]$, $\xi\in\Omega$, and where $\omega\in\R$ is the frequency variable.
We consider frequency distributions according to a probability space $(\R, \B, g(\omega)\txtd\omega)$ with an arbitrary density function $\|g\|_{L^1(\R)} = 1$. Let us point out that replacing $g\txtd\omega$ with the Dirac distribution $\delta_0$, and choosing $\beta =1$ reduces the model again to \eqref{eq:graphopmckean}.

\subsection{\new{Homogeneous case}}\label{subsec:freqhomo}
For simplicity and in order to compare the result to the existing literature, let us first omit the additional network variable, i.e., we set $A=\id$ and consider initial data $\rho_0$ independent of $\xi$. We assume that $\int_U \rho_0(x,\omega) \txtd x = 1$ for all $\omega\in\R$ and as a result it holds that $\int_U \rho(t,x,\omega) \txtd x  = 1$ for all $\omega\in\R$, $t>0$.
Analogous to the extension to the entropy for graphop interactions \eqref{eq:entropyxi}, we can define the frequency-averaged relative entropy
\begin{equation*}
	\overline{H}(\rho|\rho_\infty) := \int_\R \int_U \rho \log(\frac{\rho}{\rho_\infty}) \txtd x  g\txtd\omega.
\end{equation*}
For each $\omega\in\R$ the transport term $\partial_t \rho = -\omega \partial_x(\rho)$ conserves the relative (spatial-)entropy  $H(\rho(
\omega)|\rho_\infty)$ (as defined in \eqref{eq:entropy}), hence the computation is completely analogous to the proof of Theorem~\ref{thm:heterodecaygraphop}. Specifically, the Vlasov term \eqref{eq:freqvlasov} corresponds to \eqref{eq:vlasovgraphop} with the all-to-all coupling graphop $A_g\rho := \int_\R \rho(\omega) g(\omega) \txtd\omega$ on the probability space $(\R, \B, g\txtd\omega)$. This amounts to a straightforward relabeling of the frequency variable $\omega$ as the network variable $\xi$. With the only difference being that it is not required for $\supp g$ to be $\R$, thus solutions and the steady state $\rho_\infty$ are naturally only relevant on the support of $g$.

The graphop $A_g$ can be expressed with the graphon $W(\omega,\tilde{\omega}) = 1$ with $\|W\|_{L^2(g\txtd\omega)}=1$. Thus Corollary~\ref{cor:heterodecay} (together with the spectral estimate \eqref{eq:graphopHS}) yields $g$-independent global stability for
\begin{equation}\label{eq:kappafreq}
	\kappa < \frac{2\pi^2}{L^2\beta  \| \partial_{xx} D \|_{L^\infty(U)} } =: \kappa_0(\beta),
\end{equation}
then the solution $\rho$ to equation \eqref{eq:graphopmckean} is exponentially stable with the decay estimate
\begin{equation*}
	\overline{H}(\rho(t)|\rho_\infty) \leq  \txte^{-\alpha t} \overline{H}(\rho_0|\rho_\infty), \quad t\geq 0,
\end{equation*}
where the $g$-independent decay rate is given as
\begin{equation*}
	\alpha:= \frac{4 \pi^2}{L^2\beta } - 2\kappa  \| \partial_{xx} D\|_{L^\infty(U)} > 0.
\end{equation*}
Using the CKP inequality \eqref{eq:CKP} on the product space $L^1(\R\times U, g\txtd\omega\times \txtd x ) =: L^1(g\txtd\omega \txtd x)$, this means we have the following type of estimate
\begin{equation*}
	\|\rho-\rho_\infty\|_{L^1(g\txtd\omega \txtd x )} \leq e^{-\frac{\alpha}{2} t} c,\quad t\geq 0,
\end{equation*}
for some constant $c >0$.
\begin{remark}
	Let us compare the stability estimates of \eqref{eq:kappafreq} with the established stability results in the literature with the standard setting $L=2\pi$ and $D(x)= -\cos(x)$. Sakaguchi \cite{Saka} as well as Strogatz and Mirollo \cite{strogatz91} have proven that the critical coupling strength ---  which marks the onset of synchronization phenomena and the loss of the stability of the incoherent steady state --- is given as
	\begin{equation*}
		\kappa_c(\beta) := 2\left[ \int_\R \frac{\beta^{-1}}{\beta^{-2} + \omega^2} g(\omega)\txtd\omega\right]^{-1}\stackrel{\beta \to \infty}{\longrightarrow} \frac{2}{\pi g(0)}.
	\end{equation*}
	Using H\"older's inequality, we see that
	\begin{equation*}
		\kappa_c(\beta) \geq 2\left[\int_\R g(\omega) \txtd\omega \sup_{\omega\in\R}  \frac{\beta^{-1}}{\beta^{-2} + \omega^2} \right]^{-1} = \frac{2}{\beta} > \frac{1}{2\beta} = \kappa_0(\beta),
	\end{equation*}
	where $\kappa_0(\beta)$ was defined in \eqref{eq:kappafreq}. Hence, while our results are independent of any specific choice of frequency distribution $g(\omega)\txtd\omega$, the global stability results are not sharp. It stays below the critical value for all $\beta >0$ and converges to $0$ as $\beta\to \infty$.
	\end{remark}

\subsection{\new{Heterogeneous case}}\label{subsec:hetero}
In order to consider the presence of both network structure and frequency distribution, we prove the following lemma formulated for general ``combined graphops''.
\begin{lemma}\label{lem:combined}
	Let graphop $A_i\in \mathcal{B}_{2,2}(\Omega_i)$ and underlying probability space $(\Omega_i,\mathcal{A}_i,\mu_i)$ be given for $i=1,2$ with numerical radius $n_{\mu_i}(A_i)$. Let $\rho_0(x,\cdot,\cdot) \in L^\infty(\Omega_1\times \Omega_2)$ for each $x\in U$. Consider McKean--Vlasov equations \eqref{eq:graphopmckean} with Vlasov terms of form
	\begin{equation}\label{eq:vlasovcombined}
		V[A_1,A_2](\rho)(x,\xi_1,\xi_2):= (\nabla D * A_1A_2\rho)(x,\xi_1,\xi_2).
	\end{equation}
	
	Then solutions  $\rho(t,x,\xi_1,\xi_2)$ with network variables $\xi_i\in \Omega_i$ for $i=1,2$ fulfill the global stability results of Theorem~\ref{thm:heterodecaygraphop}, setting $\mu=\mu_1\times \mu_2$ and replacing the quantity $n(A)$ with $n_{\mu_1}(A_1)n_{\mu_2}(A_2)$.
\end{lemma}

\begin{proof}
	In this setting, the relative entropy \eqref{eq:entropyxi} contains the product measure $\mu= \mu_1\times \mu_2$. Following the proof of Theorem~\ref{thm:heterodecaygraphop}, the only difference in estimating the time-derivative of the relative entropy is in estimating the interaction term in \eqref{eq:dtHhetero}.
	Denoting $f(t,\xi_1,\xi_2) := \|\rho(t,\xi_1,\xi_2)-\rho_\infty\|_{L^1(U)}$, we can estimate the second term as
	\begin{align*}
		\kappa& \int_{\Omega_2}\int_{\Omega_1}\int_{U} \rho (\Delta_x D \star (A_1A_2\rho) ) ~\txtd x~\txtd(\mu_1\times\mu_2)(\xi_1,\xi_2)\\
		&\leq \kappa \|\Delta_x D \|_{L^\infty(U)} n_{\mu_1\times \mu_2}(A_1A_2)  \|f\|^2_{L^2(\mu_1\times \mu_2)}\\
		&= \kappa \|\Delta_x D \|_{L^\infty(U)} n_{\mu_1}(A_1)n_{\mu_2}(A_2) \|f\|^2_{L^2(\mu_1\times \mu_2)}.
	\end{align*}
	The last equality can be validated via Fubini's theorem, the self-adjointness of $A_i$ in $L^2(\Omega_i)$ and the fact that $A_1A_2f = A_2A_1f$ as each operator is linear, bounded in $L^2(\Omega_i)$ and solely acts on its distinct variable.
\end{proof}
\begin{remark}
	Note that ``combined graphop'' interactions are different from multiplex networks which would correspond to $A_1 + A_2$ where both graphops act on the same network variable. In the Sakaguchi-Kuramoto model \eqref{eq:sakaguchi} we have one arbitrary graph structure and one all-to-all frequency coupling. This can be interpreted as a frequency-dependent coloring of the graph structure.
\end{remark}

With the realization of \S\ref{subsec:freqhomo} that intrinsic frequencies can be treated analogously to all-to-all coupling and with Lemma~\ref{lem:combined}, we are now able to obtain a global stability result for Sakaguchi-Kuramoto models with heterogeneous network interactions.

\begin{prop}
	Consider the Sakaguchi-Kuramoto model \eqref{eq:sakaguchi}  with arbitrary graphop $A\in B_{2,2}(\Omega)$, probability space $(\Omega,\mathcal{A}, \mu)$ and arbitrary frequency distribution $\|g\|_{L^1(\R)} = 1$. Then the global stability result of Theorem~\ref{thm:heterodecaygraphop} is fulfilled for the relative entropy 
	\begin{equation}
		\hat{H}_{g\times \mu}(\rho|\rho_\infty):= \int_\R\int_{\Omega}\int_{U}  \rho \log(\frac{\rho}{\rho_\infty}) \, \txtd x ~\txtd \mu(\xi)~g\txtd\omega .
	\end{equation}
\end{prop}
\begin{proof}
	We apply Lemma~\ref{lem:combined} with $A_1\rho = A_g\rho= \int_\R \rho(\omega) g(\omega)~\txtd\omega$, $\Omega_1 = \R$, $\txtd\mu_1 = g\txtd
	\omega$ as defined in \eqref{eq:freqvlasov} and $A_2=A$. Then, we obtain the global stability result of Theorem~\ref{thm:heterodecaygraphop} for the relative entropy $\hat{H}_{g\times \mu}(\rho|\rho_\infty)$.
\end{proof}

\begin{remark}
	The attentive reader might have noticed that for the here presented results, we do not require the frequency distribution to be absolutely continuous with respect to the Lebesgue measure. All results work for general Borel probability measures. Nonetheless, we choose to adhere to the in the literature established notation with a density function $g$ for direct comparison.
\end{remark}

The Sakaguchi-Kuramoto example has shown that the developed entropy method to prove global stability for the splay steady state is robust even when combining heterogeneous interactions with arbitrary intrinsic frequency distributions. Let us add that the versatility of entropy functionals has also recently been showcased by providing explicit stability estimates in the large coupling strength regime \cite{PM22}.

\section{\new{Graph examples}}\label{sec:examples}
In this section we consider solutions to \eqref{eq:graphopmckean} for explicit graph interaction structures and apply the established global stability results.
\subsection{\new{Spherical graphop}}
	Let us consider the \emph{spherical graphop} \cite{BaSz20}. It is an operator defined as
	\begin{equation*}
		A: L^2(\SS^2, \mu) \to L^2(\SS^2, \mu), \quad (A\rho)(\xi):= \int_{\xi^\perp} \rho ~\txtd\nu_\xi(\txi),
	\end{equation*}
	where $\SS^2:= \{ \xi \in \R^3 : |\xi|_2 = 1\}$, $\mu$ is the uniform probablity measure on $\SS^2$  and the integration takes place along the $\xi$-equator, defined as $\xi^
	\perp:= \{\txi \in \mathbb{S}^2 \mid \xi^T\txi = 0 \}$. For each $\xi\in \SS^2$, the measure $\nu_\xi$ denotes the uniform probablity measure on the (1-dim) submanifold $\xi^\perp$.

	As no density function exists with respect to a $\xi$-independent measure, this is a graphop that has no graphon representation. Furthermore, as the degree of each $\xi$ is not finite, it is also not representable as a graphing, but rather a more general graphop located ``in-between'' graphons and graphings \cite{BaSz20}. For discussion on its induced finite network structure and resulting numerical stability estimates, we refer to \cite{GJKM}.
	
	\begin{prop}\label{prop:spherical}
	Consider the McKean--Vlasov equation \eqref{eq:graphopmckean} with a spherical graphop in the Vlasov term. Then solutions are globally stable, provided
	\begin{equation}
		\kappa < \frac{2 \pi^2}{L^2 \| \Delta_x D \|_{L^\infty(U)}  },
	\end{equation}
	with the decay rate estimate $\hat{\alpha}(A):= \frac{4 \pi^2}{L^2} - 2\kappa  \| \Delta_x D\|_{L^\infty(U)}.$
	\end{prop}
	\begin{proof}
	As a preparation, consider $f\in L^2(\SS^2, \mu)$, then, for each $\xi\in\SS^2$, we can find a (non-unique) parameterization of $f|_{\xi^\perp}$ as 
	\begin{equation}
		f_\xi(\tau) := f(\cos(\tau) v_1^\xi + \sin(\tau) v_2^\xi),\quad \tau\in [0,2\pi),
	\end{equation}
	where the vectors $\{\xi, v_1^\xi,v_2^\xi\}$ form an orthonormal basis of $\R^3$ for each $\xi\in\SS^2$. Further, we can transform each element of $\SS^2$ into spherical coordinates. We denote $\xi = \xi(\phi,\theta)$ with $\phi\in [0,\pi)$ and $\theta\in [0,2\pi)$ and transformation factor $|\sin(\theta)|$.
	In order to apply Theorem~\ref{thm:heterodecaygraphop} we show $\|A\|_{2\to 2} = 1$:

	With the considerations from above, we have
	\begin{align*}
		\|Af\|^2_{L^2(\SS^2)} &= \int_{\SS^2} \left(\int_{\xi^\perp} f~ \txtd\nu_\xi(\txi)\right)^2 \txtd\mu(\xi) = \int_{\SS^2} \left(\frac{1}{2\pi}\int_0^{2\pi} f_\xi(\tau)~\txtd\tau\right)^2 \txtd\mu(\xi)\\
		&\leq \int_{\SS^2} \frac{1}{2\pi}\int_0^{2\pi} f_\xi(\tau)^2~\txtd\tau~\txtd\mu(\xi) \\
		&= \frac{1}{2\pi} \int_0^{2\pi} \frac{1}{4\pi}\int_0^{\pi}
		\int_0^{2\pi} f_{\xi(\phi,\theta)}^2(\tau) |\sin(\theta)|~\txtd\tau ~\txtd\theta~\txtd\phi \\
		&= \frac{1}{2\pi} \int_0^{2\pi} \|f\|^2_{L^2(\SS^2)}~\txtd\phi = \|f\|^2_{L^2(\SS^2)}.
	\end{align*}
	For the inequality above, we used Cauchy-Schwarz. The second to last equality holds, as for any fixed $\phi \in [0,2\pi)$ the inner two integrations exactly integrate $f^2$ once over $\SS^2$. Self-adjointness can be shown in a similar way using the spherical coordinates and resulting symmetries. As one can validate that $A$ is a Markov graphop, i.e.\ satisfying $A1_{\SS^2} = 1_{\SS^2}$, it holds that $\|A\|_{2\to 2} = 1$.
	
	To estimate the convergence of solutions to the IVP \eqref{eq:graphopmckean} with a spherical graphop coupling, we can now directly apply Theorem~\ref{thm:heterodecaygraphop} with $n(A)\leq \|A\|_{2\to2} \leq 1$.
\end{proof}
	
	\begin{remark}
		The stability result of Proposition~\ref{prop:spherical} is identical to the case of all-to-all coupling given in Proposition~\ref{prop:homodecay}. We point out that the spherical graphop describes rather sparse interactions compared to an all-to-all coupling. Thus, it is worth investigating whether incorporating additional graphop observables into the analysis could clarify or sharpen the stability results. However, this is beyond the scope of the current study. We refer to future research and first numerical considerations of \cite{GJKM} which show improved convergence compared to the all-to-all coupled case. It is worth noting that  there is significant potential to develop reliable numerics by leveraging the particular graph structure of the dynamics \cite{bohle22}.
	\end{remark}
\subsection{Graphon Examples}

\subsubsection*{Erd\"os-R\'enyi random graph}
For the start let us mention Erd\"os-R\'enyi random graphs which take the simple graphon form $W(x,y) = p$ with $p\in(0,1)$. Then, we can apply Corollary~\ref{cor:heterodecay} for solutions to \eqref{eq:graphopmckean} with the Erd\"os-R\'enyi random graphop in the Vlasov interaction term. In particular \eqref{eq:Hdecaygraphon} provides the decay rate
	\begin{equation*}
	\hat{\alpha}(W):= \frac{4 \pi^2}{L^2} - 2\kappa  \| \Delta_x D\|_{L^\infty(U)} p > 0,
\end{equation*}
for solutions in relative entropy, given that the interaction coefficient fulfills $\kappa < \frac{2 \pi^2}{L^2 \| \Delta_x D \|_{L^\infty(U)} p }$. When compared to the homogeneous interaction case this decay aligns with the intuition, as the mean-field interaction strength $\kappa$ is simply reduced to $\kappa p$.

\subsubsection*{Power law random graphs}\label{subsec:powerlaw}

Let us now consider power law random graphs, as constructed in\cite{borgs2019}. They are important examples of intermediately sparse graphs that correspond to unbounded $L^p$ functions which the standard $L^\infty$ graphon convergence theory for dense graphs cannot handle.

To introduce power law graphs, let a set of $[N]$ vertices be given with $N\in\N$. For distinct indices $i,j\in[N]$, $i\neq j$, the vertices are connected, i.e.\ $A^{ij}=1$, with the probability
\begin{equation*}
	p(i,j) = \min\{1, N^\beta (ij)^{-\alpha}\},\quad \alpha \in (0,1), \beta \in (2\alpha -1 , 2\alpha).
\end{equation*}
This results in a superlinear expected number of edges and an expected edge density $N^{\beta - 2\alpha}$. To construct the empirical graphon $W^{(N)}_{\alpha,\beta}(\xi,\txi)$ and the finite particle interaction term  \eqref{eq:mckean}, one has to include the rescaling factor $r_N = N^{\beta - 2\alpha}$, see \cite{borgs2019,med2019}. Consequently, such graphs converge  to the \emph{power law graphon} for $N\to \infty$ (in the cut metric \cite{borgs2019} or graphop action sense \cite{BaSz20} which are equivalent in this case), denoted as:
\begin{equation}
	\label{eq:powerlawgraphon}
	W_\alpha(\xi,\txi):= (1-\alpha)^2 (\xi \txi)^{-\alpha},\quad \xi,\txi\in[0,1],\alpha \in (0,1).
\end{equation}
With the choice $\Omega = [0,1]$ and the Lebesgue measure $\mu = \lambda$, we represent the associated graphop for $\alpha \in (0,1)$ as:
\begin{align}\label{eq:powerlawgraphop}\begin{aligned}	
	A_\alpha:=A_{W_\alpha}: &\quad L^\infty([0,1], \lambda) \to L^1([0,1], \lambda)\\
&\quad f \mapsto (A_\alpha f)(\xi):= \int_{[0,1]} W_\alpha(\xi,\txi)  f(\txi) ~\txtd\txi.
	\end{aligned}
\end{align}
An increase in $\alpha\in (0,1)$ leads to a stronger localization of the power law graph around the origin $(0,0)$. 
While $W_\alpha$ is unbounded, for each fixed $\alpha\in(0,1)$ it is an $L^p([0,1]^2)$ graphon for $p\in [1,\frac1\alpha)$. Due to the bound \eqref{eq:graphonoperatornorm}, the associated graphop can be extended to an operator 
\begin{equation}
A_\alpha\in B_{p^*,p}(\Omega) \text{ for each } p \in [1,\textstyle\frac1\alpha)
\end{equation}
with $\|A_\alpha\|_{p^*\to p} \leq \|W_\alpha\|_{p}$.

In the case $\alpha \in (0,\frac12)$ it follows that the power law graphon satisfies $n(A) = \|A_\alpha \|_{2\to 2} \leq \|W_\alpha\|_{2} = \frac{(1-\alpha)^2}{(1-2 \alpha)}<\infty$ and Corollary~\ref{cor:heterodecay} provides the decay
\begin{equation*}
	\hat{H}(\rho(t)|\rho_\infty) \leq \txte^{-\hat{\alpha}(W_\alpha) t} \hat{H}(\rho_0|\rho_\infty),\quad t\geq 0,
\end{equation*}
with $\hat{\alpha}(W_\alpha):=  \frac{4 \pi^2}{L^2} - 2\kappa \| \Delta_x D\|_{L^\infty(U)} \frac{(1-\alpha)^2}{(1-2 \alpha)} > 0$ as long as the condition
\begin{equation*} \kappa < \frac{2 \pi^2(1-2 \alpha)}{L^2 \| \Delta_x D \|_{L^\infty(U)} (1-\alpha)^2}
\end{equation*}
is satisfied.

For fixed $\alpha \in [\frac12,1)$ it only follows that  $\|A_\alpha\|_{p^*\to p} \leq \|A_\alpha\|_{2\to p}<\infty$ with $p<\frac1\alpha \leq 2$ is guaranteed. In fact, as discussed in Section~\ref{subsec:noL2onA}, this is not sufficient to prove exponential decay with our method and entropy $\hat{H}(\rho|\rho_\infty)$. 

Finally, even modifying the measure $\tmu$ in the entropy, as defined in \eqref{eq:entropymod}, cannot resolve the issues for power law graphops with $\alpha>\frac12$. We prove this in the following lemma by showing that the necessary estimate \eqref{eq:CAestimate} (discussed in \S~\ref{subsec:noL2onA}) does not hold true for any $\tmu$ with $\supp \tmu = [0,1]$.
\begin{lemma}\label{lem:tmu}
	Let the probability space $([0,1], \mathcal{B}([0,1]), \lambda)$ with the power law graphop $A_\alpha$, as defined in \eqref{eq:powerlawgraphop} for $\alpha\in (\frac12,1)$ be given. Then no probability Borel measure $\tmu$ with $\supp \tmu = [0,1]$ exists such that 
	\begin{equation*}
		n_{\tmu}(A_\alpha) = \sup \{(A_\alpha f,f)_{L^2([0,1],\tmu)} \mid f \in L^\infty([0,1],\tmu), \|f\|_{L^2([0,1],\tmu)} = 1 \} < \infty.
	\end{equation*}

\end{lemma}
\begin{proof}
For any such measure $\tmu$ and $\rho\in L^\infty([0,1])$, we  have the following identity
	\begin{equation}\label{eq:numrangeA}
		\int_{[0,1]} \rho(\xi) (A_\alpha \rho)(\xi)  \txtd\tmu(\xi) = (1-\alpha)^2 \left(\int_{[0,1]} \xi^{-\alpha} \rho(\xi) \txtd\tmu(\xi)\right) \left(\int_{[0,1]} \xi^{-\alpha} \rho(\xi) \txtd\xi \right).
	\end{equation}
Now, for $n\in\N$ define the sequence
\begin{equation*}
	\rho_n(\xi):= 
	\begin{cases}
		\tmu([0,\frac{1}{n}])^{-\frac12},& \xi \in [0,\frac{1}{n}],\\
		0,& \xi \in (\frac1n,1].
	\end{cases}
\end{equation*}
 Then $\rho_n \in L^\infty(\tmu)$ and $\|\rho_n\|_{L^2(\tmu)} = 1$ for all $n\in \N$. 
In order for $n_{\tmu}(A_\alpha)<\infty$ to hold, it is necessary that:
\begin{equation}\label{eq:rhonbound}
	\int_{[0,1]} \rho_n(\xi) (A \rho_n)(\xi)~\txtd\tmu(\xi) = \tmu([0,\frac1n])^{-1} \int_{[0,\frac1n]} \xi^{-\alpha}~\txtd\tmu(\xi) ~n^{\alpha - 1} \leq C
\end{equation}
for all $n\in \N$. But as  $\xi^{-\alpha} \geq n^{\alpha}$ for all $\xi\in[0,\frac1n]$, it follows that
\begin{equation*}
\tmu([0,\frac1n])^{-1} \int_{[0,\frac1n]} \xi^{-\alpha} ~\txtd\tmu(\xi) ~n^{\alpha - 1} \geq  n^{2\alpha -1}. 
\end{equation*}
Plugging this estimate back in \eqref{eq:rhonbound} leads to the condition $n^{2\alpha -1} \leq C$, which needs to be satisfied for all $n\in\N$. But as $2\alpha -1 >0$ by assumption, this cannot be satisfied.
\end{proof}

\begin{remark}
     We have discussed the limitations of the entropy method in \S\ref{subsec:noL2onA}. Here, we computed the limitations explicitly on the examples of power law graphops. In summary, the main culprit stems from the lack of regularization mechanisms in the equation in the network variable. Solutions evolve only according to a spatial diffusion term and a spatial drift, which do not directly improve the regularity of solutions in the network variable over time. To achieve improved regularity through the equation's dynamics, one could consider modifying the interaction term such that the two variables are coupled more directly.
\end{remark}

\subsection*{Conclusion}
In this work, we investigated graphop McKean--Vlasov equations as a mean-field formulation of a system of interacting stochastic McKean (or noisy Kuramoto-type) differential equations with diverse heterogeneous interaction patterns. The resulting solutions depend not only on space and time but also on an additional network variable $\xi$. We have shown existence of classical solutions for each fixed $\xi$ and finite $\xi$-averaged entropy. We further applied the entropy method to show global stability of the chaotic steady state, provided the graphop is bounded in $L^2$ and the dynamic's interaction strength does not exceed a graphop-dependent threshold. 
Crucially, the method achieves explicit decay rates and can deal with graphs of dense, intermediately dense and sparse structures of unbounded degree. This has been demonstrated on various prototypical examples. We have also extended the results to the closely related Sakaguchi-Kuramoto model with frequency distribution and heterogeneous interactions, highlighting the method's robustness with respect to model variations. We have discussed the limitations and demonstrated them explicitly on the examples of power law graphops. With the provided results, we hope to showcase a rather general graph limit theory of graphops is well suited to be incorporated in established PDE methods. 
\bibliography{23kw} 
\bibliographystyle{plain} 

\end{document}